\documentclass[a4paper]{article}

\usepackage{a4wide}
\usepackage{enumerate}

\usepackage{color}
\usepackage[]{hyperref}
\usepackage[utf8]{inputenc}

\usepackage[T1]{fontenc}
\usepackage{verbatim} 

\usepackage{amsmath, amssymb, amsfonts,amsthm}

\usepackage{dsfont}
\usepackage{marvosym} 

\usepackage{graphicx}
\usepackage{caption}
\usepackage{float}
\restylefloat{figure}
\usepackage{subfig}

\usepackage{paralist}

\usepackage{MnSymbol} 

\advance\textheight by 2cm
\theoremstyle{plain}
\newtheorem{theorem}{Theorem}[section]

\newtheorem{lemma}[theorem]{Lemma}
\newtheorem{remark}[theorem]{Remark}
\newtheorem{cor}[theorem]{Corollary}

\newtheorem{prop}[theorem]{Proposition}

\theoremstyle{definition}
\newtheorem{definition}[theorem]{Definition}
\newtheorem{example}[theorem]{Example}

\newcounter{stp}

\DeclareMathOperator{\real}{Re}

\usepackage{authblk}

\title{More than one Author with different Affiliations}
\author[1]{Martin Adler\thanks{maad@fa.uni-tuebingen.de}}
\author[1]{Waed Dada\thanks{wada@fa.uni-tuebingen.de}}
\author[2]{Agnes Radl \thanks{agnes.radl@math.uni-leipzig.de}}
\affil[1]{Mathematisches Institut, Universität Tübingen\\
 Auf der Morgenstelle 10\\ D-72076 Tübingen, Germany}
\affil[2]{Mathematisches Institut, Universität Leipzig\\Augustusplatz 10, 04109 Leipzig, Germany}

\begin{document}

\title{A semigroup approach to the numerical range of operators on Banach spaces}

\date{}
\maketitle


{\bf \large Abstract.}
We introduce the {\em numerical spectrum} $\sigma_n(A)\subseteq\mathbb{C}$ of an (unbounded) linear operator $A$ on a Banach space $X$ and study its properties. 
Our definition is closely related to the numerical range $W(A)$ of $A$ and always yields a superset of $W(A)$. 
In the case of bounded operators on Hilbert spaces, the two notions coincide. 
However, unlike the numerical range, $\sigma_n(A)$ is always closed, convex and contains the spectrum of $A$. 
In the paper we strongly emphasise the connection of our approach to the theory of $C_0$-semigroups. 
\\

\vspace{0.5cm}
\vspace{0.2cm}

\section{Introduction}

The spectrum $\sigma(A)$ is the most important invariant of a 
linear operator $A$ on a Banach space X. It is a closed set of complex 
numbers and is invariant under similarity of operators. However  it is, in some 
sense, a very weak invariant since, e.g., $\sigma(A) = \{0\}$ does not 
imply $A = 0$.

Already in 1918, O. Toeplitz  \cite{to18} and F. Hausdorff \cite{ha19}  introduced a stronger invariant for  $A$ acting on the finite dimensional space $X = \mathbb C^n$ with Euclidean norm $||\cdot||$ induced by the standard inner product $< \cdot, \cdot > $ as
$$W(A) := \{< Ax,x> : x\in X, ||x||=1 \},$$
calling it {\em  Wertevorrat einer Bilinearform}, later  {\em numerical range } of $A$.
This set is invariant  only under unitary equivalence and has some very nice 
properties such as
\begin{enumerate}
\item $ W(A)$ contains $\sigma(A)$,
\item  $W(A)$ is closed,
\item  $W(A)$ is convex,
\item $W(A) = \{0\}$ if and only if $A= 0$. 
\end{enumerate}
In the following,  M. H. Stone \cite{st32}  extended this concept to bounded operators on arbitrary Hilbert 
spaces by the same formula. This numerical range remains a unitary 
invariant, but loses properties (1) and (2) while (3) and (4) still hold. \\
Later,  J. R. Giles  and G. Joseph \cite{gj74} generalized the numerical range to unbounded operators on Hilbert spaces. In this case one loses the properties (1) and (2) in an even more drastic way, see  \cite{bo75} and \cite{cj76}. \\
Due to the lack of an inner product, it was not so immediate how to 
define a numerical range for operators on Banach spaces. This 
was achieved by Lumer in 1961 \cite{lg61} and Bauer in 1962 \cite{bf62}. However, by this generalization one loses  all the properties (1), (2), and (3), while the numerical range remains  an isometric invariant.

Our goal in this paper is twofold. First, we modify the definition of the numerical range in order to 
obtain a set of complex numbers, called \emph{numerical spectrum} of a linear 
operator on a Banach space, being an isometric invariant and satisfying all the properties 
(1),(2),(3) and (4). 

Second, we show how this set characterizes semigroups $(T(t))_{t\ge0}$  satisfying 
an estimate of the form
$||T(t)|| \le M e^{t\omega} \;\mbox{for} \; t \ge0\;\mbox{  for some}\; \omega\in \mathbb R\; \mbox{but with}\; M = 1.$\\
Such {\em quasi-contractive semigroups} (more precisely, $\omega$-contractive) are important in many applications, 
e.g., in order to obtain stability in the Trotter product formula (see \cite[Cor. III.5.8]{en00}). \\
We now recall here the  basic  tools for our approach. For an operator $A$ with domain $D(A)$ on a Banach space and some fixed $\theta\in[0, 2\pi)$, we  consider the ``rotated'' operator $A_{\theta}:= e^{-i\theta}A$ with $D(A_{\theta})=D(A)$. If this rotated operator $A_{\theta}$  generates a $C_0$-semigroup, we denote it  by   $(T_{\theta}(t))_{t\ge0}$. Analogously, for $\omega\in \mathbb R$ we rotate the half plane $\mathbb C_{\omega}:=\{\lambda\in \mathbb C: \mbox{Re} \; \lambda>\omega\}$,  by an angle of $\theta$ and obtain the rotated half plane as 
 $$H_{\theta, \omega}:= e^{i\theta}\cdot \mathbb C_{\omega}= \{e^{i\theta}\cdot \lambda: \mbox{Re}\; \lambda >\omega\}.$$

 We define the distance between $\lambda\in H_{\theta, \omega}$ and $\mu \in \partial H_{\theta,\omega}$ as
$$d(\lambda, \partial H_{\theta,\omega}):= \inf\{|\lambda-\mu|: \mu\in \partial H_{\theta,\omega}\}.$$
Combining the Hille-Yosida Theorem (\cite[Theorem II.3.5]{en00}) with the Lumer-Phillips Theorem (\cite[Theorem II.3.15]{en00} we obtain a series of  equivalences.

\begin{theorem} \label{goodeq11}
For  a closed and densely defined operator $(A,D(A))$ on a Banach space $X$ and some $\theta\in [0,2\pi)$, $\omega\in \mathbb R$, the following statements are equivalent.
\begin{enumerate}
\item $(A_{\theta},D(A_{\theta}))$ generates an $\omega$-contractive $C_0$-semigroup $(T_{\theta}(t))_{t\ge0}$.
 \item $e^{i\theta}(\omega, +\infty)\subseteq \rho(A)$ and $\|R(e^{i \theta}\lambda, A)\|\le \frac{1}{ \lambda-\omega}$ for all $\lambda >\omega$.

\item  $H_{\theta, \omega}\subseteq \rho(A)$  and $|| R( \lambda, A)||\le \frac{1}{d( \lambda , \partial H_{\theta,\omega})}$ for all  $\lambda \in H_{\theta, \omega}$.
\item  $||(e^{i\theta} \lambda -A)x||\ge (\lambda-\omega) ||x||$ for all  $ \lambda >\omega,\; x\in D(A)$, and $(e^{i\theta}\lambda-A)$ is surjective for some (hence all) $\lambda>\omega$.

\item For each $x\in D(A)$ there exists $j(x) \in \mathfrak J(x)$ \footnote[1]{For every $x\in X$, the {\em duality set} $\mathfrak J(x)$ is defined as \begin{center}$\mathfrak J(x):=\{j(x) \in X^\prime: \langle x,j(x)\rangle =||x||^2=||j(x)||^2\}$, see \cite{en00}, p.~87.\end{center}} such that $$\mbox{Re}\;e^{-i\theta}\langle Ax,j(x)\rangle \le \omega||x||^2$$ and $(e^{i\theta}\lambda-A)$ is surjective  for some (hence all) $\lambda>\omega$.

\end{enumerate}
\end{theorem}

\section{The numerical spectrum}

We now use the characterizations of Theorem \ref{goodeq11} in order to define first the {\em numerical resolvent set} and then the {\em numerical spectrum} of a given operator.

\begin{definition}\label{7de}
Let $(A, D(A))$ be a closed and densely defined operator on a Banach space X. Then the {\em numerical resolvent set} of A is 
\begin{eqnarray*}\rho_{n}(A) :=\bigcup \big \{H_{\theta, \omega}: &H_{\theta, \omega}\subseteq \rho(A)&  \mbox{is an open half plane such that } \\  &&\|R(\lambda, A)\|\le \frac{1}{d(\lambda,\partial H_{\theta,\omega})} \;\; \forall \lambda\in H_{\theta, \omega} \big \}. \end{eqnarray*}
The  complementary set 
$$\sigma_n(A):= \mathbb C\backslash \rho_n(A)$$
is called the {\em  numerical spectrum} of A.
%
\end{definition}

\begin{remark}\label{re}
\begin{enumerate}
\item
As a first observation we note that the  open half plane $H_{\theta, \omega}$ in the above definition belongs to $\rho_n(A)$. Therefore, $\rho_n(A)$ is empty or a union of open half planes, hence open, while $\sigma_n(A)$ is $\mathbb C$ or an  intersection of the  closed half planes $\mathbb C\backslash H_{\theta, \omega }$, hence closed and convex. 
In particular, the numerical spectrum $\sigma_n(A)$ is never  empty.

\item   Using the statements (1) and (3) of Theorem \ref{goodeq11}, we see that $\lambda \in \rho_n(A)$ if and only if there exists $ (\theta, \omega )$ such that $ \lambda \in H_{\theta, \omega} $ and $ A_{\theta} $ generates an $\omega$-contractive $C_{0}$-semigroup. Therefore, $\rho_n(A)$, and hence $\sigma_n(A)$, is characterized by a generator property of the rotated operators $A_{\theta}$.
 
\end{enumerate}
\end{remark}

The numerical spectrum $\sigma_n(A)$  is closely related to the spectrum and the so-called {\em numerical range}  of A. 
Before discussing this relationship,  we state some simple properties of the numerical spectrum.

\begin{prop} \label{equa} Let $(A, D(A))$  be a closed and  densely defined operator on a Banach space $X$. Then we have the following. 
\begin{enumerate}
\item $\sigma_n(A)$ is closed and convex.

\item  $\sigma_n (\alpha A+\beta)= \alpha \sigma_n(A)+\beta$ for all complex numbers $\alpha$ and $\beta$.
\item $\sigma_n(A)= \sigma_n(U^{-1}AU)$ for all isometric isomorphisms $U$ on $X$. 

\end{enumerate} 
\end{prop}

\begin{proof}
\begin{enumerate}
\item This has been noted in Remark \ref{re}.(1).

\item We show $\rho_n(\alpha A+\beta)=\alpha  \rho_n(A)+\beta$.  If $\alpha =0$, the assertion is clear. \\
$\supseteq$: If $\alpha \not= 0$ take $\lambda \in (\alpha  \rho_n(A)+\beta)$, or   $\alpha^{-1}(\lambda-\beta)\in \rho_n(A)$. Then there exists an open half plane $H_{\theta,\omega}$ such that $\alpha^{-1}(\lambda-\beta) \in H_{\theta,\omega} \subseteq \rho(A)$  and $$||R(\alpha^{-1}(\lambda-\beta), A)||\le \frac{1}{d(\alpha^{-1}(\lambda-\beta), \partial H_{\theta,\omega})}.$$
Therefore,
$$ |\alpha|\; ||R(\lambda, \alpha A+\beta)||= ||R(\alpha^{-1}(\lambda-\beta), A)||  \le \frac{1}{d(\alpha^{-1}(\lambda-\beta), \partial  H_{\theta,\omega})}.$$
Since
 $$\frac{1}{d(\alpha^{-1}(\lambda-\beta), \partial  H_{\theta,\omega})}=\frac{1}{|\alpha^{-1}| d(\lambda, \alpha \partial H_{\theta,\omega}+\beta)}= \frac{|\alpha|}{d(\lambda, \alpha \partial H_{\theta,\omega}+\beta)},$$
we obtain
$$||R(\lambda, \alpha A+\beta)|| \le\frac{1}{ d(\lambda, \alpha \partial  H_{\theta,\omega}+\beta)}.$$
Since  $\lambda \in (\alpha H_{\theta,\omega}+ \beta) \subseteq \rho(\alpha A+\beta)$, it  follows that  $\lambda \in \rho_n(\alpha A+\beta)$.\\
$\subseteq$: Let $\alpha \not= 0$ and $\lambda \in \rho_n(\alpha A+\beta)$. Then there exists an open half plane $H_{\theta,\omega}$ such that \\ $\lambda \in H_{\theta,\omega} \subseteq \rho(\alpha A+\beta)= \alpha \rho(A)+\beta$ and $||R(\lambda, \alpha A+\beta)||\le \frac{1}{d(\lambda, \partial  H_{\theta,\omega})}$.\\
Hence,
$$
\frac{1}{d(\lambda, \partial  H_{\theta,\omega})} \ge ||R(\lambda, \alpha A+\beta)||= |\alpha^{-1}|\cdot ||R(\alpha^{-1}(\lambda-\beta),  A)||.$$
From 
\begin{eqnarray*}
d(\lambda, \partial H_{\theta,\omega}) &=& \inf \{ | \lambda-z|: z\in \partial H_{\theta,\omega}\}\\
&=& \inf \{| \;\alpha \;[\alpha^{-1}(\lambda-\beta)-\alpha^{-1}(z-\beta)]\;| : z\in \partial H_{\theta,\omega}, \alpha, \beta \in \mathbb C \}\\
&=& |\alpha|\cdot d\big(\alpha^{-1}(\lambda-\beta),\alpha^{-1}(\partial H_{\theta,\omega}-\beta)\big)
\end{eqnarray*}
we obtain the equality 
$$\frac{1}{d(\lambda, \partial  H_{\theta,\omega})} =\frac{1}{|\alpha|\cdot d(\alpha^{-1}(\lambda-\beta), \alpha^{-1}(\partial  H_{\theta,\omega}-\beta))}.$$
Thus 
$$||R(\alpha^{-1}(\lambda-\beta),  A)||\le\frac{1}{ d\big(\alpha^{-1}(\lambda-\beta), \alpha^{-1}(\partial H_{\theta,\omega}-\beta)\big)} ,$$
and hence $\lambda \in (\alpha \rho_n(A)+ \beta)$.

\item The claim follows since $ \rho(U^{-1} A U)= \rho(A)$  and 
\begin{center}$||R(\lambda,  U^{-1} A U)||= ||U^{-1} R( \lambda , A)U|| = ||R(\lambda, A)||.$\end{center}
\end{enumerate}\vspace{-0.8cm}
\end{proof}
We now  show that $\sigma(A)\subseteq\sigma_n(A)$. Moreover, if $A$ is bounded, then $\sigma_n(A)$ is compact.
\begin{prop}\label{spectrum}
 For every  closed and densely defined operator  $(A, D(A))$ on a Banach space $X$, we have the following assertions.
\begin{enumerate}
 \item $\sigma(A)\subseteq \sigma_n(A)$. 
\item If $A$ is a bounded operator, then $\sigma_n(A)\subseteq \{\lambda\in \mathbb C: |\lambda|\le ||A||\}$.

\end{enumerate}
\end{prop}
\begin{proof}
\begin{enumerate}
\item By Definition \ref{7de}, it is clear that $\rho_n(A)\subseteq \rho(A)$, hence $\sigma(A)\subseteq \sigma_n(A) $.
\item Since $A_{\theta}$ is a bounded operator on $X$ with $||A_{\theta}||= ||A||$ for all $\theta\in[0,2\pi)$, we have $$||T_{\theta}(t)||=\left\| \sum_{n=0}^{\infty} \frac{t^n A_{\theta}^n}{n!}\right\|\le \sum_{n=0}^{\infty} \frac{t^n ||A||^n}{n!} = e^{t||A||}.$$ Thus for each $\theta\in [0,2\pi)$, $A_{\theta}$ generates a $||A||$-contractive semigroup.  The  assertion now follows from Remark \ref{re}.(2).
\end{enumerate}\vspace{-0.5cm}
\end{proof}

The following theorem relates our numerical spectrum of an operator $(A,D(A))$ to the numerical range  defined as
$$W(A):=\{\langle Ax, j(x)\rangle : x\in D(A), ||x||=1, j(x)\in \mathfrak J(x)\},$$
and studied, e.g., in  \cite{bo73}, \cite{bo75}. 
Here, and in the sequel, $\overline {co} M$ denotes the closure of the convex hull of the set $M$. 

\begin{theorem}\label{union} Let $(A, D(A))$ be a closed and densely defined operator on a Banach space $X$. Then we have 
\begin{center}
$\sigma_n(A)= \overline{co}\{\langle Ax,j(x)\rangle: x\in D(A), \|x\|=1,  j(x)\in \mathfrak J(x)\} \cup \sigma_r(A)$ \footnote[1]{We call $\sigma_r(A):=\{\lambda\in \mathbb C: rg(\lambda-A)$ is not dense in $X$\} the {\em residual spectrum} of $A$. It coincides  with the point spectrum  $\sigma_p(A^\prime)$ of the adjoint $A'$ of $A$, see \cite[Prop.~IV.1.12]{en00}. }.
\end{center}
\end{theorem}

\begin{proof} 
$\supseteq$ :  Since $\sigma_n(A)$ is closed and convex, it suffices to take $$\lambda\in \{\langle Ax,j(x)\rangle: x\in D(A), \|x\|=1,  j(x)\in \mathfrak J(x)\} \cup \sigma_r(A).$$
If  $\lambda\in \sigma_r(A)$, then $\lambda \in \sigma_n(A)$, since $\sigma_r(A)\subseteq \sigma(A)\subseteq \sigma_n(A)$.  If  $\lambda = \langle Ax,j(x)\rangle$ for some $ x\in D(A), \|x\|=1,  j(x)\in \mathfrak J(x)$.

Let $\sigma_n(A) \subseteq  H_{\theta,\omega} $ be a closed half plane.  
It follows that    $e^{-i\theta}A-\omega$ is the generator of a contractive $C_0$-semigroup. Hence $$\mbox{Re}\; e^{-i\theta}\langle Ax, j(x)\rangle \le \omega\; \mbox{ for all}\;  x\in D(A),\; ||x||=1,\; j(x)\in \mathfrak J(x).$$
Now 
$$\mbox{Re}\; e^{-i\theta} \lambda= \mbox{Re} \;e^{-i\theta}  \langle Ax, j(x)\rangle \le \omega.$$
Therefore  $\lambda\in \sigma_n(A)$. \\
$\subseteq$: Take $\lambda \notin \overline{co}\{\langle Ax,j(x)\rangle: x\in D(A), \|x\|=1,  j(x)\in \mathfrak J(x)\}$ and $\lambda \notin \sigma_r(A)$. There exists an open half plane $H_{\theta,\omega}$ such that $\lambda \in H_{\theta,\omega}$ and  $$\{\langle Ax,j(x)\rangle: x\in D(A), \|x\|=1,  j(x)\in \mathfrak J(x)\} \subset \mathbb C\backslash H_{\theta,\omega}.$$ Rotating  yields
$$\mbox{Re}\; \langle e^{-i\theta}Ax,j(x)\rangle \le \omega \;\; \mbox{for all} \;x\in D(A), ||x||=1, j(x)\in \mathfrak J(x), $$
i.e., $e^{-i\theta}A -\omega$ is dissipative. Since $\lambda \notin \sigma_r(A)$, then  $e^{-i\theta} \lambda-\omega\notin\sigma_r(A_{\theta}-\omega)$ and  $A_{\theta}$ is surjective for some (hence all) $\lambda>\omega$.  By Theorem \ref{goodeq11}, $A_{\theta}$ is the generator of an $\omega$-contraction semigroup and by Remark \ref{re}.(2) one has $\lambda \in \rho_n(A)$. 
\end{proof}
\begin{remark} 
Combining the above theorem with Proposition \ref{equa}(1) it follows that the union 
$$\overline{co}\{\langle Ax,j(x)\rangle: x\in D(A), \|x\|=1,  j(x)\in \mathfrak J(x)\} \cup \sigma_r(A)$$ is automatically a closed and convex set. 
\end{remark}

We now list some cases in which $\sigma_n(A)$ takes a simpler form. 
For example, for  a bounded operator $A$ we do not need $\sigma_r(A)$ in the above characterization of $\sigma_n(A)$.

\begin{cor}\label{closure1}
For a bounded operator  $A$ on a Banach space X, we have 
$$ \sigma_n(A)= \overline{co}\{\langle Ax,j(x)\rangle:  x\in X, \|x\|=1,  j(x)\in \mathfrak J(x)\}.$$
\end{cor}

\begin{proof} 
$\supseteq$: This inclusion follows from Theorem \ref{union}.\\
$\subseteq$: If $\lambda \notin \overline{co} \{\langle Ax,j(x)\rangle: j(x)\in \mathfrak J(x), x\in X, \|x\|=1\}$,  there exists an open half plane $ H_{\theta,\omega}\subseteq \mathbb C$ with $\overline{co} \{\langle Ax,j(x)\rangle: j(x)\in \mathfrak J(x), x\in X, \|x\|=1\} \subset \mathbb C\backslash H_{\theta,\omega}$ and $\lambda \in H_{\theta,\omega}$. 
 By Theorem \ref{goodeq11} we have that $e^{-i\theta}A$ generates an $\omega$-contractive $C_0$-semigroup and Remark \ref{re}.(2) implies $\lambda\in \rho_n(A)$.
\end{proof}
For bounded operators on a Hilbert space $\mathbb H$ we do not need to take the convex hull  since convexity follows automatically.

\begin{cor}\label{closure}
If $A$ is a bounded operator on a Hilbert space $\mathbb H$ with scalar product \\$<\cdot,\cdot>_{\mathbb H}$, then we have  
$$ \sigma_n(A)=\overline{W(A)},$$
where $W(A) $ is the numerical range of A.
\end{cor}
\begin{proof}   In a Hilbert space one has $J(x)=\{x\}$, see \cite[II.3.26.(iii)]{en00}. Moreover,  the set \\$\{< Ax,x>_{\mathbb H}:  x\in H, \|x\|=1\}$ is always convex by Theorem 1.1-2 in \cite{gr97} or Theorem (Toeplitz-Hausdorff) in \cite{gr70}.
\end{proof}

\begin{remark}\label{re3} In Section $5.4$ we give some examples illustrating cases in which  the residual spectrum $\sigma_r(A)$ contributes to $\sigma_n(A)$.
\end{remark}
The numerical spectrum, unlike the spectrum, characterizes various properties of the operator $(A, D(A))$.

\begin{prop} \label{kompakt}Let $(A, D(A))$ be a closed, densely defined operator on a Banach space $X$. Then we have the following assertions.
\begin{enumerate}\label{com}
  \item $ A$ is bounded if and only if $\sigma_n(A)$ is compact.
\item  $A=0$ if and only if $\sigma_n(A)=\{0\}$.
\end{enumerate}
\end{prop}

\begin{proof} 
\begin{enumerate}
\item $\Rightarrow$: By Proposition \ref{spectrum} and Remark \ref{re} the claim is clear.\\
$\Leftarrow$: If $\sigma_n(A)$ is compact,  there exists a ball $\overline{B_r(0)}$ with center $0$ and radius $r\in \mathbb R$ such that $\sigma_n(A)\subseteq \overline{B_r(0)}$. Then $\sigma_n(A)$ is contained in the strip $$\Pi_{-r, r}:=\{\lambda \in \mathbb C:  -r\le \mbox{Re} \;\lambda\le r\}.$$
By Proposition \ref{311}.4 below, $A$ generates a group. Moreover there exists a sector $\Sigma^{0}_{\delta}+z\;\footnote[1]
{For $ 0<\delta<\frac{\pi}{2}$ and $\theta\in [0,2\pi)$, $z\in \mathbb C$ we define $\Sigma^{\theta}_{\delta}+z:=\{ z+ e^{-i\theta}\;\lambda \in \mathbb C: |arg \lambda |\le \delta\}$.}$
such that $\sigma_n(A) \subset \Sigma^{0}_{\delta}+z$. Again by Proposition \ref{311}.3 below,  $A$ generates an analytic semigroup. Consequently, $A$ generates an analytic group and therefore  
$$X=T(t)X\subset D(A) \;\; \mbox{for all } \;\; t>0.$$
Hence $D(A)=X$ and $A\in \mathcal L(X)$.

\item  $\Rightarrow$: If $A$ is zero, then it is clear that $\sigma_n(A)=\{0\}$. \\
$\Leftarrow$: If $\sigma_n(A)=\{0\}$, then $A$ is a bounded operator on $X$ by assertion (1). Therefore $T(z)= e^{zA}$ is a holomorphic function on $ \mathbb C$. Since  $||T(z)|| \le 1$ for all $z\in \mathbb C$, by Liouville's theorem we obtain that $T(z)$ is constant for all $z\in \mathbb C$. Hence  $T(t)= Id$ for all $t\ge0$. Thus, we obtain $A=0$.
\end{enumerate}\vspace{-0.5cm}
\end{proof}
Since each closed, convex set $\emptyset\not=\Omega \subset\mathbb C$ is the intersection of closed half planes containing $\Omega$, we obtain the following simple classification of $\Omega$ depending on the number of half planes needed for this inclusion.

\begin{lemma} \label{manyform}
Each closed, convex subset $\emptyset \not=\Omega \subseteq \mathbb C$ belongs to (at least) one of the following classes:

\begin{enumerate}
\item $\Omega=\mathbb C$, 
\item $\Omega$ is a half plane $\mathbb C\backslash H_{\theta,\omega}=\{ e^{i\theta} \;\lambda : Re \lambda \le \omega\}$ for some $ \omega\in \mathbb R$ and $\theta\in [0,2\pi)$.
\item $\Omega$ is contained in a sector $\Sigma^{\theta}_{\delta}+z:=\{ z+ e^{-i\theta}\;\lambda \in \mathbb C: |arg \lambda |\le \delta\}$ for some $ 0<\delta<\frac{\pi}{2}$ and $\theta\in [0,2\pi)$, $z\in \mathbb C$. 
\item $\Omega$  is contained in a strip  $\Pi_{\omega_1,\omega_2}^{\theta}:=\{e^{-i\theta} \;\lambda\in \mathbb C: \omega_1 \le Re \lambda\le\omega_2\}$ for $\omega_1\le \omega_2 \in \mathbb R$ and $\theta \in [0, 2\pi)$. 

\item $\Omega$ is compact.
\end{enumerate}
\end{lemma}

Since the numerical spectrum $\sigma_n(A)$ is closed and convex, it belongs to (at least) one of the classes $(1), \ldots, (5)$.  This allows the following (not disjoint) classification of the operator $(A,D(A))$  with respect to generator properties.

\begin{prop}  \label{311} Every closed, densely defined operator $A$ on a Banach space $X$ satisfies (at least) one of the following conditions.
\begin{enumerate}
\item  $\sigma_n(A)= \mathbb C$, i.e., there is no $ (\theta, \omega) \in[0,2\pi) \times \mathbb R$ such that $ A_{\theta}$ generates an $\omega$-contractive semigroup. 
\item  $\sigma_n(A) \subseteq \mathbb C\backslash H_{\theta,\omega}$, i.e.,  there exists $\theta\in [0, 2\pi)$ and $\omega\in \mathbb R $ such that $ A_{\theta}$ generates an $\omega$-contractive semigroup.

\item  $\sigma_n(A) \subset \Sigma^{\theta}_{\delta}+z$, i.e., there exists $\theta\in [0, 2\pi)$ and $ \omega\in \mathbb R$ such that $ A_{\theta}$ generates an analytic  $\omega$-contractive semigroup.

\item $\sigma_n(A)\subset \Pi_{\omega_1,\omega_2}^{\theta}$, i.e., there exists $ \theta \in [0,2\pi)$ and $\omega_1< \omega_2 \in \mathbb R$ such that $ A_{\theta}$ generates a $\omega_2$-contractive  semigroup and $ -A_{\theta}$ generates a $-\omega_1$-contractive semigroup. 
\item  $\sigma_n(A)$ is contained in a compact set, i.e.,  $ A$ generates a norm continuous semigroup.

\end{enumerate}
 As a special case of (4) we remark that  $A_{\theta}$ generates an isometric group if and only if $\sigma_n(A)\subseteq i \mathbb R$.
\end{prop}

\begin{figure}[h]
    \begin{center}
        \subfloat[$\sigma_n(A) \subseteq \mathbb C\backslash H_{\theta,\omega}$]{%
            \label{fig:first}
\raisebox{0.7cm}{
            \includegraphics[width=0.40\textwidth]{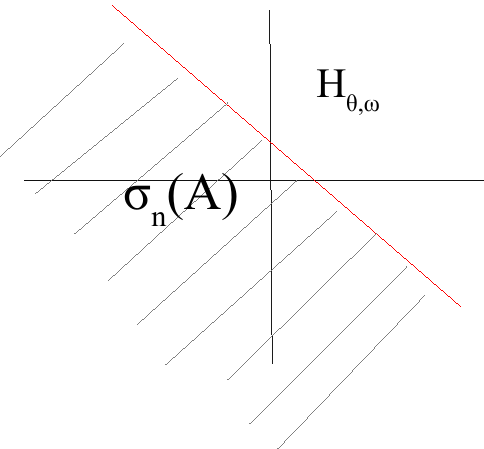}}
        }%
\hspace{1cm}
        \subfloat[$\sigma_n(A) \subset \Sigma^{\theta}_{\delta}+z$]{
           \label{fig:second}
           \includegraphics[width=0.40\textwidth]{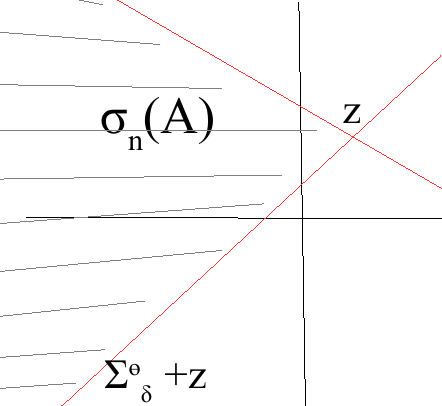}}
\end{center}
\end{figure}
\begin{figure}[h]
    \begin{center}
        \subfloat[$\sigma_n(A) \subseteq \Pi_{\omega_1,\omega_2}^{\theta}$]{%
            \label{fig:first}
            \includegraphics[width=0.40\textwidth]{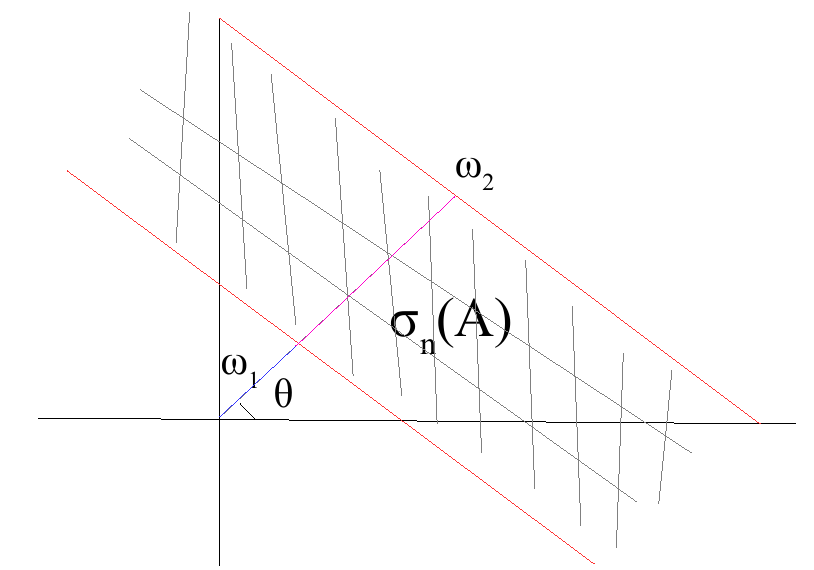}
}
\hspace{1cm}               
\subfloat[$\sigma_n(A)$ is compact ]{%
           \label{fig:second}
           \includegraphics[width=0.40\textwidth]{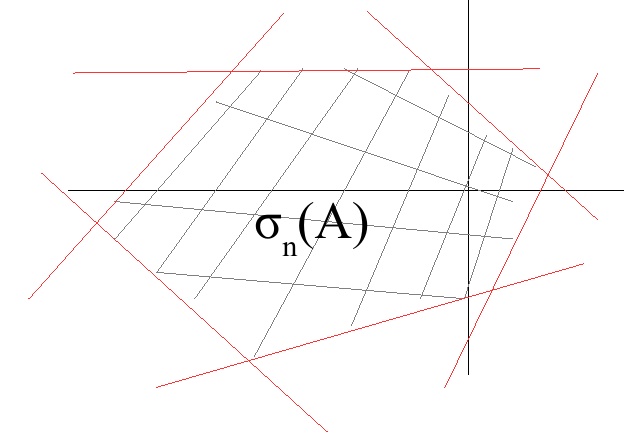}
        }
     \end{center}
\vskip 0.2in
\caption{$\sigma_n(A)$ is contained in different sets  determined by few half planes}
   \label{fig:subfigures}
\end{figure}
\newpage

\section{Numerical growth bound, numerical spectral bound and numerical radius}

We now introduce constants related to the numerical spectrum $\sigma_n(A)$ analogously to the growth bound $\omega_0(A)$ and spectral bound $s(A)$ of a generator $(A,D(A))$ (see \cite[Def.~I.5.6]{en00} and \cite[Def.~IV.2.1]{en00}).


\begin{definition} \label{defnumgrowthbound1} Let $(A, D(A))$ be a closed and densely defined operator on a Banach space X. For $\theta\in [0,2 \pi)$ we consider $A_{\theta}:= e^{-i\theta}A$ and, in case it exists,  the corresponding semigroup $(T_{\theta}(t))_{t\ge0}$. Then we call
$$\omega_n^{\theta}(A):= \inf\{\omega \in \mathbb R:  ||T_{\theta}(t)||\le e^{t \omega} \;\mbox{for all}\;  t\ge0\}$$
the {\em numerical  growth bound} of $A$ corresponding to the angle $\theta\in [0,2\pi)$, where we set\\ $\inf \emptyset := +\infty$.
\end{definition}
\begin{prop}\label{pro.III.1.2}
\begin{enumerate}
\item  In contrast to the growth bound,  the numerical bound, \\if not $+\infty$,  is attained, i.e.,  for each $\theta \in [0,2\pi)$ we have $$ \omega_n^{\theta}(A)=\min\{\omega\in \mathbb R:  ||T_{\theta}(t)||\le e^{t \omega}  \mbox{for all}\;  t\ge0\}$$
and therefore $|| T_{\theta}(t)||\le e^{t \omega_n^{\theta}(A)}$ for all $t\ge 0$.
\item If $A_{\theta}$ generates the semigroup $(T_{\theta}(t))_{t\ge0}$, then $\omega_n^{\theta}(A)\ge \omega_0(A_{\theta})$ and   $$ -\infty<\omega_n^{\theta}(A)\le +\infty.$$
\end{enumerate}
\end{prop}

\begin{proof}
\begin{enumerate}
\item Let $(\omega_m)_{m\in \mathbb N}$ be a monotone sequence with $\omega_m\searrow \omega_n^{\theta}(A)$ as $m \to \infty$ for some fixed $\theta$. Then we have  $$||T_{\theta}(t)||\le e^{t \omega_m}\;\; \mbox{for all} \;\;m\in \mathbb N, t\ge0.$$
Hence 
$$e^{-t\omega_n^{\theta}(A)}||T_{\theta}(t)||\le e^{-t\omega_n^{\theta}(A)} e^{t \omega_m}\;\; \mbox{for all} \;\;m\in \mathbb N, t\ge0,$$
or $$e^{-t\omega_n^{\theta}(A)}||T_{\theta}(t)||\le \inf_{m\in \mathbb N}{e^{t(\omega_m-\omega_n^{\theta}(A))}}  \;\mbox{for all} \;\;t\ge0,$$
and thus
$$ ||T_{\theta}(t)||\le e^{t\omega_n^{\theta}(A)} \;\;\mbox{for all} \;\;t\ge0.$$
\item It is clear by definition that $\omega_n^{\theta}(A)\ge \omega_0(A_{\theta})$ and, since $\sigma_n(A)\not=\emptyset$, it follows that   $ \omega_n^{\theta}(A)>-\infty$. In  Example \ref{ex1} below we  see that $\omega_n^{\theta}(A)= +\infty$ may occur.
\end{enumerate}\vspace{-0.5cm}
\end{proof}

Analogously to the spectral bound we now define the numerical spectral bound of  $A$.

\begin{definition} \label{numspbou}For a closed, densely defined operator $(A,D(A))$ on a Banach space $X$ 
the {\em numerical spectral bound} of $A$  is  
$$s_n(A) := \sup\{\mbox{Re} \;\lambda : \lambda\in \sigma_n(A)\}.$$
Moreover, for $\theta\in [0,2\pi)$ we define $$ s_n^{\theta}(A):=s_n(A_{\theta}),$$
the numerical spectral bound of $A_{\theta}$. 
\end{definition}

In the following proposition we prove for a generator $(A,D(A))$ the equality between $s_n^{0}(A)$ and $w_n^{0}(A)$.
\begin{prop} \label{equality3} Let $(A, D(A))$ be a closed and densely defined operator on a Banach space $X$ and $(T_{0} (t))_{t\ge0}$ be the semigroup generated by $A$. Then
$$ s_n^{0}(A)\stackrel{(1)}= \omega_n^{0}(A)\stackrel{(2)}=\sup_{t>0}\frac{1}{t} \log|| T_{0}(t)||\stackrel{(3)}=\lim_{t\searrow0}\frac{1}{t} \log|| T_{0}(t)||.$$
\end{prop}

\begin{proof} The inequality $(1)$ holds by definition and the equivalences in Theorem \ref{goodeq11}.\\
From  $||T_0(t)||\le e^{ t \omega_n^{0}(A) }$ for all $t\ge 0$ we obtain
$$ \sup_{t>0}\frac{1}{t} \log|| T_0(t)||\le \omega_n^{0}(A).$$

On the other hand, assume that $\mu:= \sup_{t>0}\frac{1}{t} \log|| T_0(t)|| < \omega_n^{0}(A)$ and take  $\phi\in (\mu, \omega_n^{0}(A))$, i.e., for all $t\ge0$ we have 
\begin{gather}
\frac{1}{t} \log|| T_0(t)||\le \phi.   \label{II.1.1} \tag{$*$}
\end{gather} 
Since $\phi<\omega_n^{0}(A) $,   there exists $t>0$ such that $||T_0(t)||>e^{\phi t}$. Therefore $$\frac{1}{t} \log|| T_0(t)||> \phi,$$  contradicting the inequality \eqref{II.1.1}. \\

Showing the third equality, we suppose without loss of generality that,  $\omega_n^{0}(A)=0$, i.e., $||T_0(t)||\le 1$ for all $t\ge0$. By the semigroup property, the map $t\longmapsto ||T_0(t)||$ is monotonically decreasing and so is the map $t\longmapsto \frac{1}{t} \log ||T_0(t)||$. Hence, $$ \lim_{t\searrow0}\frac{1}{t} \log|| T_{0}(t)||=\sup_{t>0}\frac{1}{t} \log|| T_{0}(t)||.$$
\vspace{-1cm}
\end{proof}

Compare the above equalities to 
$$ s(A)\le \omega_0(A)=\inf_{t\ge0} \frac{1}{t} \log||T_{0}(t)||=\lim_{t\to \infty}\frac{1}{t} \log||T_{0}(t)||$$
from \cite[IV.2.2]{en00}. Combining this with Proposition \ref{equality3} we obtain 
$$ s(A)\le \omega_0(A)\le s_n^{0}(A)= \omega_n^{0}(A).$$

\begin{definition}\label{numrad} 
For a bounded operator $A$ on a Banach space $X$, the {\em numerical radius} of A is
$$r_n(A):=\sup\{|\lambda|: \lambda\in \sigma_n(A)\}.$$
\end{definition}
From Proposition \ref{spectrum} we conclude that  $r(A)\le r_n(A)$. 
Moreover, using Proposition  \ref{equality3} we obtain $$r_n(A)= \sup_{\theta\in [0,2\pi)} \omega_n^{\theta} (A)=\sup_{\theta\in [0,2\pi)} s_n^{\theta}(A).$$

\section {Applications}

Our  concept of numerical spectrum allows  to generalize the main result in \cite{kl11} from Hilbert  to Banach spaces. 

\begin{theorem}
Let $(T(t))_{t\ge0}$ be a  $C_0$-semigroup with generator $(A, D(A))$.  The following are equivalent.
\begin{enumerate}[(a)]
\item \label{III.2.3(a)}There exists $\epsilon>0$ such that $||T(t)||\le e^{-t\epsilon} $ for all $t\ge0$. i.e., $\sigma_n(A)\subseteq \mathbb C_{-\epsilon}$.
\item \label{III.2.3(b)}There exists $\epsilon>0$ such that Re $\langle Ax,j(x)\rangle \le -\epsilon ||x||^2$ for all  $ x\in D(A)$ and $ j(x)\in \mathfrak J(x)$.
\end{enumerate}
\end{theorem}
\begin{proof}
$(\ref {III.2.3(a)}) \Rightarrow (\ref{III.2.3(b)}): $ By assumption, the $C_0$-semigroup $(e^{t\epsilon} T(t))_{t\ge0}$ is a contraction semigroup. Hence Re $\langle (A-\epsilon)x,j(x)\rangle \le0$ for all $x\in D(A)$ and $j(x)\in \mathfrak J(X)$. Thus, 
\begin{align*}
	\mathrm{Re} \langle Ax,j(x)\rangle \le -\epsilon ||x||^2
\end{align*}
for all  $ x\in D(A)$ and $ j(x)\in \mathfrak J(x)$.\\

$(\ref {III.2.3(b)}) \Rightarrow (\ref{III.2.3(a)}): $ There exists $M\ge1$ and $\omega\in \mathbb R$ such that $||T(t)|| \le M e^{t\omega}$ for all $t\ge0$.  By the assumption,  $(A+\epsilon, D(A))$ is dissipative and  $\lambda \in \rho(A) $  for all $\lambda > \omega$. By the Lumer-Phillips Theorem \cite[Theorem II.3.15]{en00} we obtain $||T(t)||\le e^{-t\epsilon}$ for all $t\ge0$ and the assertion follows. 
\end{proof}

In the next proposition we extend Hildebrandt's Theorem \cite[Thm. 6.3]{sh03}, \cite[Thm. 2.4]{hi66} to bounded operators on Banach spaces. It yields the convex hull of the spectrum $\sigma(A)$ by varying the norm of the space $X$. For this purpose we use the following notation. \\
Let $A\in \mathcal L(X)$. 
We write $\sigma_n^{||\cdot||}(A)$  to indicate that the numerical spectrum is computed with respect to $||\cdot||$. Clearly, the numerical spectrum changes when we switch to a different, but equivalent norm on $X$.

\begin{prop}
If $A$ is a bounded operator on a Banach space $X$ with norm $||\cdot||_{X}$, then 
$$ \overline{co}\;\big(\sigma(A)\big)= \bigcap\{ \sigma_n^{||\cdot||}(A): ||\cdot|| \mbox{ equivalent to }\; ||\cdot||_{X} \}.$$
\end{prop}
\begin{proof}
Since $\sigma(A)$ is independent of the chosen norm and always contained in the (convex) numerical spectrum, the inclusion $"\subseteq"$ is obvious. \\
To show  the converse inclusion it suffices  to consider $A_{\theta}$ for $\theta=0$. 
 Since $A_{0}$ is bounded, its spectral and growth bound coincide, i.e., $s(A_{0})=\omega_{0}(A_{0})$.  For every $\omega>s(A_{0})$ there exists $M\ge 1$ such that $$||e^{t A_0}||\le M e^{t\omega} \; \mbox{for} \;t\ge0.$$
We  define an equivalent norm by
$$ ||| x|||:=\sup_{t\ge0} || e^{-t\omega} e^{t A_0}x|| \;\; \mbox{for}\; x\in X.$$
For this norm we obtain
$$ |||e^{t A_0}x|||\le e^{t\omega}|||x|||\;\; \mbox{for} \;t\ge0,\; x\in X,$$
which implies $s_n^{0}(A_{0})\le \omega$ on the Banach space $(X,|||\cdot|||)$. Since this holds for all $\omega> s(A_{0})$, we conclude 
$$ s(A_{0})=\inf \{ s_n^{0}(A_{0}):||\cdot|| \mbox{ equivalent to }\; ||\cdot||_{X}\}.$$
The standard rotation argument for $\theta\in[0,2\pi)$ finishes the proof.
\end{proof}

\section{Examples}
We now discuss various examples for the numerical spectrum.
The Schur decomposition theorem, see \cite{fz05}, guarantees that any square matrix  transformed by unitary similarity into an upper triangular form. So we start with  upper triangular matrices.

\begin{example}
Let $X:=\mathbb C^2$ be endowed with $||\cdot||_p$, $1\le p\le \infty$ and 
$A:=\begin{pmatrix} 0&1\\0&0\end{pmatrix}\in \mathcal L(X)$.\\

A straightforward calculation using the generated semigroup $\begin{pmatrix} \begin{pmatrix} 1&t\\0&1\end{pmatrix}\end{pmatrix}_{t\ge0}$ shows that   the numerical spectrum of $A$ is a closed disk centered at $0$ with numerical radius 
$$
r_n(A)=\begin{cases}
1\hspace{2cm}\;\;\;\;\;\;\;\; \mbox {for} \; p=1,\infty\\
\big(\frac{p-1}{p}\big)^{\left(1-\frac{1}{p}\right)} \big(\frac{1}{p}\big)^{\frac{1}{p}} \;\;\;\; \mbox{for} \; 1<p<\infty.
\end{cases}
$$

\end{example}

\begin{example}
Let $X:= \mathbb C^{2}$ be endowed with $||\cdot||_1$ and consider $A:=\begin{pmatrix} 1&1\\0& -1\end{pmatrix}$. 
Then the generated semigroup  is given by 
$$ T(t)= e^{tA}= \begin{pmatrix} e^t & \frac{1}{2}(e^t -e^{-t})\\ 0& e^{-t}\end{pmatrix}, t\ge 0.$$
By Proposition \ref{equality3}  for  $\theta=0$ we have
\begin{eqnarray*}
s_n^{0}(A)= \omega_n^{0}(A)&=& \lim_{t\searrow0}\frac{1}{t} \log|| T(t)||\\
&=&\lim _{t\searrow 0} \frac{\log ||T(t)|| -\log||T(0)||}{t}\\
&=& \left. \frac{d}{dt} \log ||T(t)|| \right|_{t=0}=1. \\
\end{eqnarray*}

Then we have 
 $\sigma_n(A)\subseteq \{z\in \mathbb C: \mbox{Re}\; z\le 1 \}$. To determine exactly the location of the numerical spectrum, consider the upper triangular matrix  $A + Id = \begin{pmatrix} 2&1 \\0&0\end{pmatrix} =:B $. From Proposition \ref{equa} one has $\sigma_n(A)= -1+\sigma_n(B)$.\\
For $x:=\binom{(1-s)e^{i\theta}}{s e^{i\varphi}}$ with $||x||_1=1$ and $j(x):= \binom{e^{-i\theta}}{e^{-i\varphi}}$ with $||j(x)||_{\infty}=1$ for $s\in [0,1]$ and $\theta,\varphi \in [0,2\pi)$  one has $||x||_1= ||j(x)||_{\infty}=\langle x,\varphi\rangle =1$ and for $s\in [0,1]$

\begin{eqnarray*}
\langle Bx,j(x)\rangle &= &\big\langle \begin{pmatrix} 2&1 \\0&0\end{pmatrix}\binom{(1-s)e^{i\theta}}{s e^{i\varphi}}, \binom{e^{-i\theta}}{e^{-i\varphi}} \big\rangle \\
&=& \big\langle \binom{2(1-s)e^{i\theta} +s e^{i\varphi}}{0}, \binom{e^{-i\theta}}{e^{-i\varphi}}\big\rangle\\
&=& 2+s(-2+e^{i(\varphi -\theta)}).
\end{eqnarray*}
Then
$\sigma_n(B)=\{2+s(-2+e^{i(\varphi -\theta)}) : s\in[0,1], \theta,\varphi\in [0,2\pi)\}$. Since the set 
$\{-2+e^{i(\varphi -\theta)} : \theta,\varphi\in [0,2\pi)\}$ is a circle with center at $(-2,0)$ and radius $1$,  the multiplication with $s\in [0,1]$ gives us all the lines from $0$ to the circle as in  Figure $(a)$.

Since the numerical spectrum is convex,  $\sigma_n(B)$ is a circular cone with vertex at $0$.
 By shifting  the numerical spectrum of $\sigma_n(B)$ by $1$ to the left, one obtains the numerical spectrum of $A$. 
\end{example}

\begin{figure}[h]
    \begin{center}
        \subfloat[$\sigma_n(B)$ on $(\mathbb C^2, ||\cdot||_1)$]{%
            \label{fig:first}
            \includegraphics[width=0.40\textwidth]{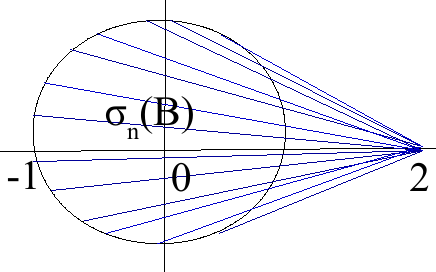}
        }%
\hspace{1cm}
        \subfloat[$\sigma_n( A)$ on $(\mathbb C^2, ||\cdot||_1)$]{%
           \label{fig:second}
           \raisebox{0.2cm}{
\includegraphics[width=0.40\textwidth]{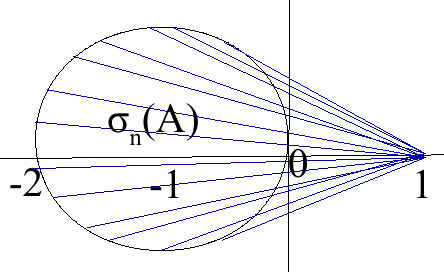}   }
           }
\end{center}
\caption{}
\end{figure}

\subsection{Translation semigroups}

In the next example we see that for unbounded operators A the spectrum $\sigma(A)$ is not  always contained in $ \overline{W(A)}$,  while it is contained in $\sigma_n(A)$.
\begin{example}\label{ex5.3}
Consider $(T(t))_{t\ge0}$ the right translation semigroup of isometries on the Hilbert space  $H= L^2(\mathbb R_{+})$ with generator  $Af=-f^\prime$  on the domain
$$ D(A):=\{ f\in H^1(\mathbb R_{+}) :  f(0)=0\}.$$
Clearly, the operator $A$ is dissipative, and hence $\sigma(A)\subset \mathbb{C}_{-}\footnote[1]{In the following, $\mathbb{C}_{-}$ denotes the closed left half plane, i.e., $\mathbb{C}_{-}:=\{\lambda\in\mathbb{C}:\real{\lambda}\leq 0\}$.}$. 
The left translation semigroup $(S(t))_{t\ge0}$ is generated by the adjoint $A'$ of $A$ where $A' f=f^\prime$ on the domain $D(A')= H^1(\mathbb R_{+})$. 

Since for each $\lambda\in\mathbb{C}$ with $\real\lambda<0$ the function $\epsilon_{\lambda}(t):= e^{\lambda t}$ is an eigenfunction of $A'$ to the eigenvalue $\lambda$ and since $\sigma_p(A')=\sigma_r(A)$, we obtain  that $\sigma(A)=\mathbb{C}_{-}$. 
From \cite[p. 175]{tr05} we know that the numerical range is $W(A)=i\mathbb R$, while it follows from Theorem \ref{union} and from Proposition \ref{spectrum} that  $\sigma_n(A)=\mathbb C_{-}$.

\end{example}
We now present  an example showing that $\sigma_n(A)=\mathbb C$ is possible.
\begin{example} \label{ex1} 
Define the left translation semigroup $(T(t))_{t\ge0}$ with a jump on \\$X:= L^1[-1,+\infty )$ by
\[(T(t)f)(s):=\begin{cases} 2f(s+t) & \text{if } s\in [-t,0],\\
                           f(s+t) & \text{otherwise.} \end{cases}\]
Clearly, $\|T(t)\|\leq 2$ for each $t\geq 0$. 
If $\mathds{1}_{[0,t]}$ denotes the characteristic function of the interval $[0,t]$ for $t>0$, then we have $||T(t)\mathds{1}_{[0,t]}||= 2 ||\mathds{1}_{[0,t]}||$. 
Therefore,  $(T(t))_{t\ge0}$ is a strongly continuous semigroup satisfying 
\begin{equation}\label{boundequal2}
\|T(t)\|=2 \quad \text{ for every }t>0. 
\end{equation}
In particular, the spectrum $\sigma(A)$ of the generator $(A,D(A))$ is contained in the left half plane $\mathbb C_-$, see \cite[Prop.~IV.2.2]{en00}. 
On the other hand, a simple calculation proves that for Re~$\lambda < 0$ the function 
\[f_\lambda(s)=\begin{cases}2 e^{\lambda s} & \text{if} \;\;s\in [-1,0],\\
 e^{\lambda s} &\mbox{otherwise,}\end{cases}\] 
is an eigenfunction of $A$ showing that $\mathbb{C}_{-}\subseteq\sigma(A)$, and hence $\sigma(A)=\mathbb{C}_-$. 

From \eqref{boundequal2} we conclude that there is no $\omega\in \mathbb R$ such that $\|T(t)\|\le e^{t\omega}$  for each $t\ge0$. 

Further, there is no $\theta\in[0,2\pi) $ such that $A_{\theta}$ generates an $\omega$-contractive $C_0$-semigroup for any $\omega\in \mathbb R$:
\begin{itemize}
\item The case $\theta=0$ is excluded as $||T(t)||=2$ for all $t>0$.
\item Let $\theta=\pi$. Since $\mathbb C_{-}\subseteq \sigma(A)$,  we have $\mathbb C_{+}\subseteq \sigma(-A)$, hence $A_{\pi}=-A$ is not the generator of a $C_0$-semigroup.
\item Assume that there exists $\theta\in (0,\pi)\cup (\pi, 2\pi)$  such that $A_{\theta}$ generates an $\omega$-contractive $C_0$-semigroup $(T(t))_{t\ge0}$. Then $\sigma_n(A)\subseteq \{ e^{-i\theta} \lambda: \mbox{Re}\; \lambda\le \omega\}$. However, since $i \mathbb R\subset \sigma(A)\subseteq \sigma_n(A)$, this is not possible. 
\end{itemize}
Hence there exists no $(\theta, \omega)\in [0,2\pi)\times \mathbb R$ such that $A_{\theta}$ generates an $\omega$-contractive semigroup. It follows that $\sigma_n(A)=\mathbb C$.
\end{example}

\subsection{Multiplication operators}

We show that the numerical spectrum of multiplication operators induced by a function $q$ coincides with the closed convex hull of the spectrum.

\begin{prop} \label{pro5}
Let $\Omega$ be a locally compact space, $X:=C_0(\Omega)$, and $q:\Omega\rightarrow\mathbb{C}$ a continuous function. 
For the multiplication operator $(M_q,D(M_q))$ with 
\begin{align*}
D( M_q) &:=\{ f\in X: q\cdot f\in X\},\\
M_q f&: =q\cdot f \;\; \mbox{for all } \; f\in D(M_q),
\end{align*}
we have 
$$\sigma_n(M_q)= \overline{co}\{ q(\Omega)\}.$$
\end{prop}
\begin{proof} $\supseteq $: Since by \cite[Prop.~I.4.2]{en00} and Proposition~\ref{kompakt} $\overline{q( \Omega)}= \sigma(M_q)\subseteq \sigma_n(M_q)$ and since $\sigma_n(M_q)$ is convex, we obtain $$\overline{co}\{ q(\Omega)\}\subseteq \sigma_n(M_q).$$
$\subseteq$:  Take $\lambda \in \mathbb C\backslash \overline{co}\{ q(\Omega)\}$. Then there exists a closed half plane  $  H_{\theta,\omega} \supseteq  \overline{co}\{ q(\Omega) \} $ such that $\lambda\notin H_{\theta,\omega}$. It is clear that $d(\lambda,  \overline{q(\Omega)})\ge d(\lambda, \partial H_{\theta,\omega})$, and for $z\notin H_{\theta,\omega}$ one has
\begin{eqnarray*}
\|R(z,M_q)\|&=&\| (z-M_q)^{-1}\|= \|M_{\frac{1}{z-q}}\|= \sup_{s\in \Omega} \left|\frac{1}{z-q(s)}\right| \\
&=& \frac{1}{d(z, q(\Omega))}= \frac{1}{d(z,\overline{q(\Omega)})}\le \frac{1}{d(z, \partial H_{\theta,\omega})}. 
\end{eqnarray*}
Thus, for all $z\in \mathbb C\backslash H_{\theta,\omega} \subset \rho(M_q)$ one has $||R(z, M_q)||\le \frac{1}{d(z, \partial H_{\theta,\omega})}$. \\
Hence $\lambda \in \rho_n(M_q)$.
\end{proof}
An analogous  result holds for multiplication operators on $L^p(\Omega, \mu)$ for $1\le p\le \infty$.

\begin{prop}
Let $(\Omega, \Sigma, \mu)$ be a $\sigma$-finite measure space, $X:= L^p( \Omega, \mu), 1\le p\le\infty$, and $q: \Omega \rightarrow \mathbb C$ a measurable function. 
For the multiplication operator $(M_q,D(M_q))$ with 
\begin{align*}
D(M_{q})&:=\{f\in L^p(\Omega, \mu): q\cdot f \;\in L^p(\Omega, \mu)\},\\
M_{q} f&:= q\cdot f \;\; \mbox{for all} \; f\in D(M_q),
\end{align*}
we have 
$$\sigma_n(M_q)= \overline{co}\{q_{ess}(\Omega)\}$$
where $$q_{ess}(\Omega):= \{\lambda \in \mathbb C: \mu\big (  \{s\in \Omega : |q(s)-\lambda|< \epsilon\}\big)\not= 0 \,\mbox{for all}\; \epsilon >0\}$$
is the essential range of $q$.
\end{prop}

\begin{proof}
$\supseteq $ : We know that $\sigma(M_q)= q_{ess}( \Omega)$,  see \cite[Proposition I.4.10]{en00}. By Proposition \ref{spectrum} and Proposition \ref{equa} we have  
$\overline{co}\{q_{ess}(\Omega)\}\subseteq \sigma_n(M_q)$.

$\subseteq$:  Take $\lambda \in \mathbb C\backslash  \overline{co}\{q_{ess}(\Omega)\}$. Then there exists an open half plane $ H_{\theta, \omega}$ such that \\$  \overline{co}\{q_{ess}(\Omega)\} \subseteq \mathbb C \backslash H_{\theta, \omega}$ and  $e^{-i\theta} H_{\theta, \omega}= \mathbb C_\omega$. 
It follows that $\sigma(M_q^{\theta})\subseteq \{\lambda\in \mathbb C: \mbox{Re} \lambda \le \omega\}$. 
Take $f\in  L^p(\Omega, \mu)$ and $\lambda \in H_{\theta,\omega}$. Then 
\begin{eqnarray*}
\|R(e^{-i\theta}\lambda,M_q^{\theta})f\|^p&=&\|R(\lambda,M_q)f\|^p= \left\|\frac{1}{\lambda-q}f\right\|^p\\
&=&\int_{\Omega}\left |\frac{f(s)}{\lambda-q(s)}\right|^p ds.\\
&\le& \frac{1}{(d(\lambda,q(\Omega)))^p}\|f\|^p\le \frac{1}{(d(\lambda,\partial H_{\theta,\omega}))^p}\|f\|^p.\\
\end{eqnarray*}
Thus $$\|R( \lambda ,M_q)\|\le \frac{1}{d(\lambda,\partial H_{\theta,\omega})}\;\; \forall \lambda\in H_{\theta,\omega},$$ 
and hence  $\lambda \in \rho_n(M_q)$.
\end{proof}
\subsection{Laplace operator}
\begin{example} 
Consider the Laplace operator $\triangle f=f^{\prime \prime}$ on the  Hilbert space $\mathbb H:= L^{2}(0,2\pi)$   with domain $$D(\triangle):=\{f\in  H^2(0,2\pi) : f(0)=f(2\pi)=0\}.$$
It follows from \cite[Ex.~4.6.1]{da95} that the spectrum consists only of eigenvalues given by $\sigma(\triangle)=\{-n^2: n\in\mathbb{N}, n\geq 1\}$. 
The eigenfunction corresponding to the eigenvalue $-n^2$ is given by $f_n(x)=\sin (n  x )$. 
Since $\triangle$ is a self-adjoint operator, by the spectral theorem there exists a $\sigma$-finite measure space $(\Omega,\Sigma,\mu)$ and a function $q: \Omega\rightarrow\mathbb{R}$ such that $\triangle$ is unitarily equivalent to a multiplication operator $M_q$. 
Since $\sigma(\triangle)=\sigma(M_q)=q_{ess}(\Omega)$ by  \cite[Prop. I.4.10]{en00} and  $\overline{co}\{q_{ess}(\Omega)\}= \sigma_n(M_q)$ by the previous proposition, we conclude that  $$\sigma_n(\triangle)=\sigma_n(M_q) =(-\infty, -1].$$

\end{example}
\subsection{More examples} 
We  now investigate the description of $\sigma_n(A)$ in Theorem \ref{union} more closely.
\begin{enumerate}
\item If $\sigma_n(A)=\mathbb C$, then $\sigma_r(A)$ is needed to obtain $\sigma_n(A)$ in general.
We give an example where  
 $\overline{co}\{\langle Ax,j(x)\rangle: x\in D(A), ||x||=1, j(x)\in \mathfrak J(x)\}$ is   a strip, while $\sigma_n(A)=\mathbb C$.
\begin{example}
Consider the Hilbert space $H:= L^2(\mathbb R_{+})\oplus L^2(\mathbb R_+)$ and define the generator $(A,D(A))$ by  $A:=B+C$ and $D(A):= D(B)\oplus D(B)$, where 
$ Bf:= -f^\prime $,  $Cf:= -(Id+B) f$ 
and $$D(B):=\{f\in H^1(\mathbb R_+): f(0)=0\}.$$
From Example \ref{ex5.3} we know that 
$\sigma_r(B)=\{\lambda \in \mathbb C: \mbox{Re}\; \lambda <0 \}$ and \\
$$\{\langle Bf, f\rangle: f\in D(B), ||f||=1 \} = i \mathbb R.$$ Further,  $\sigma_r(C)=\{\lambda\in \mathbb C: Re \lambda > -1\}$ and $\{\langle C f, f\rangle: f\in D(C), ||f||=1\} = -Id+i\mathbb R$. 
Thus $ \overline{co}\{\langle Af,f\rangle: f\in D(A), ||f||=1\}$ is the strip between $-Id+i\mathbb R$ and $i\mathbb R$, while the numerical spectrum is $\sigma_n(A)=\mathbb C$.
\end{example}
\end{enumerate}

In the following assume $\sigma_n(A)\not=\mathbb C$. 

\begin{enumerate}
\item[(2)] If $ \overline{co}\{\langle Ax,j(x)\rangle: x\in D(A), ||x||=1, j(x)\in \mathfrak J(x)\}$ is contained in a strip, then $\sigma_r(A)$ is needed  in general, see Example \ref{ex5.3}.
\item[(3)] Assume that $M:=\overline{co}\{\langle Ax,j(x)\rangle: x\in D(A), ||x||=1, j(x)\in \mathfrak J(x)\}$ cannot be enclosed by a strip. Then $\overline{co}\{\langle Ax,j(x)\rangle: x\in D(A), ||x||=1, j(x)\in \mathfrak J(x)\}$ coincides with $\sigma_n(A)$, and hence $\sigma_r(A)$ is not needed to compute $\sigma_n(A)$.
\begin{proof} 
For $\lambda\notin M$ there exists a closed half plane $H$ such that $M\subset H$ and $\lambda\notin H$. 
Assume that $\lambda \in \sigma_r(A)$. By \cite[Prop. II.3.14.(ii)]{en00} we obtain 
\begin{align}
	\mathbb{C}\backslash H \subset \sigma_r(A). \label{halfplaneResSpectrum}
\end{align}
Consider the family $(H_i)_{i\in I}$ of closed half planes containing $M$. 
For each $i, j\in I$ we have 
\begin{align*}
	(\mathbb C \backslash H_i) \cap (\mathbb C\backslash H_j)\neq \emptyset,  
\end{align*}
since $\overline{co}M$ cannot be enclosed by a strip. 
In particular, 
\[(\mathbb{C}\backslash H)\cap(\mathbb{C}\backslash H_i) \neq \emptyset \quad\text{ for all } i\in I, \]
and hence using \eqref{halfplaneResSpectrum} 
\[(\mathbb{C}\backslash H_i)\cap \sigma_r(A)\neq\emptyset\quad\text{ for all }i\in I.\]
Again by \cite[Prop. II.3.14.(ii)]{en00} we obtain that 
\[\mathbb{C}\backslash H_i\subseteq \sigma_r(A) \quad\text{ for all }i \in I\]
leading to $\mathbb{C}\backslash M \subseteq \sigma_r(A)$. 
Hence, $\sigma_n(A)=\mathbb{C}$ contradicting our assumption $\sigma_n(A)\neq\emptyset$. 
\end{proof}
\item[(4)] Let $\overline{co}\{\langle Ax,j(x)\rangle: x\in D(A), ||x||=1, j(x)\in \mathfrak J(x)\}$ be contained in a half strip  described by three half planes $H_1$, $H_2$ and $H_3$. A half plane argument as in $(3)$ implies $\sigma_n(A) = \overline{co}\{\langle Ax,j(x)\rangle: x\in D(A), ||x||=1, j(x)\in \mathfrak J(x)\}$. 

\end{enumerate}

\begin{figure}[h]
    \begin{center}
           \includegraphics[width=0.30\textwidth]{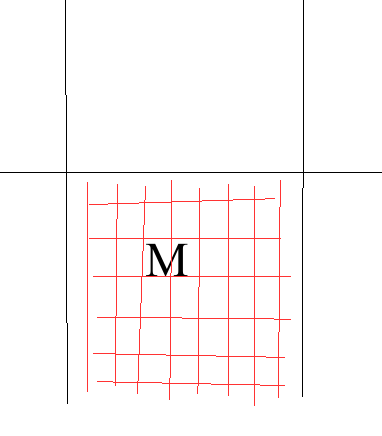}

\end{center}

\caption{}
\end{figure}
Finally we ask the following open question:  \\Is there an operator $(A, D(A))$ such that $\overline{co}\{\langle Ax,j(x)\rangle: x\in D(A), ||x||=1, j(x)\in \mathfrak J(x)\} $ is (contained in) a sector  $ \sum_{\delta}^{\theta}+z$ and  $\sigma_n(A)= \mathbb C$?

\vspace{0.5cm}

{\bf \Large Acknowledgments}\\

The  authors thank Prof.~Rainer Nagel for many valuable discussions  and remarks.

\end{document}